\documentclass[preprint,12pt]{elsarticle}

\journal{arXiv}

\usepackage{enumitem} 
\usepackage{algorithm,algpseudocode}
\usepackage{amsmath,amssymb}
\usepackage{amsthm}
\usepackage{bm}

\usepackage[hidelinks]{hyperref} 
\usepackage{xurl}
\usepackage{siunitx} 
\setlist{itemsep=-4pt,topsep=0pt}
\algnewcommand{\algorithmicgoto}{\textbf{go~to}}%
\algnewcommand{\Goto}[1]{\algorithmicgoto~Step~\ref{#1}}%

\newcommand{\Continue}{\textbf{continue}}
\newcommand{\To}{\textbf{to}}

\newcommand{\ForEach}[1]{\For{\textbf{each} #1}}
\newcommand{\EndForEach}{\EndFor{}~\textbf{each}}
\algnewcommand{\IIf}[1]{\State\algorithmicif\ #1\ \algorithmicthen}
\algnewcommand{\ElseIIf}[1]{\algorithmicelse\ #1}
\algnewcommand{\ElseI}[1]{\algorithmicelse\ #1}
\algnewcommand{\EndIIf}{\unskip\ \algorithmicend\ \algorithmicif}
\algdef{SE}[DOWHILE]{Do}{doWhile}{\algorithmicdo}[1]{\algorithmicwhile\ #1}%

\theoremstyle{plain}
\newtheorem{theorem}{Theorem}[section]

\newtheorem{lemma}[theorem]{Lemma}

\theoremstyle{definition}
\newtheorem{example}[theorem]{Example}
\newtheorem{definition}[theorem]{Definition}
\newtheorem{remark}[theorem]{Remark}

\newcommand \ie {\textit{i.e.}}
\newcommand \eg {\textit{e.g.}}

\newcommand \rev {\mathop{\rm rev}}
\newcommand \totient{\phi}  
\newcommand \MoebiusFn{\mu}  
\newcommand\CoeffHeight{{\rm H}}
\newcommand\notdiv{\nmid}
\newcommand\rad{\mathop{\rm rad}}

\newcommand \CC {{\mathbb C}}
\newcommand \FF {{\mathbb F}}
\newcommand \NN {{\mathbb N}}
\newcommand \QQ {{\mathbb Q}}

\newcommand \ZZ {{\mathbb Z}}

\def\res{\mathop{\rm res}}

\def\tfrac #1#2{{\textstyle\frac{#1}{#2}}}

\def\cocoa{\mbox{\rm
C\kern-.13em o\kern-.07 em C\kern-.13em o\kern-.15em A}}
\def\apcocoa{\mbox{\rm
A\kern-0.13em p\kern -0.07em C\kern-.13em o\kern-.07 em C\kern-.13em
o\kern-.15em A}}

\begin{document}

\begin{frontmatter}

\title{Cyclotomic Factors and LRS-Degeneracy}

\author[TUK]{John Abbott}
\ead{abbott@dima.unige.it}
\affiliation[TUK]{organization={Fachbereich~Mathematik, RPTU~University~Kaiserslautern--Landau},
	organizationsep={ }, 
    addressline={Gottlieb-Daimler-Strasse},
    city={Kaiserslautern},
    postcode={D-67663},
    country={Germany}}

\author[UPA]{Nico Mexis}
\ead{nico.mexis@uni-passau.de}
\affiliation[UPA]{organization={Fakult\"at f\"ur Informatik und Mathematik, Universit\"at~Passau},
    addressline={Innstrasse},
    city={Passau},
    postcode={D-94032},
    country={Germany}}

\begin{abstract}
  We present three new, practical algorithms for polynomials in $\ZZ[x]$:
  one to test if a polynomial is cyclotomic, one to determine which cyclotomic
  polynomials are factors, and one to determine whether
  the given polynomial is LRS-degenerate.  A polynomial is \emph{LRS-degenerate} iff
  it has two distinct roots $\alpha, \beta$ such that $\beta = \zeta \alpha$
  for some root of unity $\zeta$.  All three algorithms are based
  on ``intelligent brute force''.  The first two produce the indexes
  of the cyclotomic polynomials; the third produces a list of degeneracy orders.
  The algorithms are implemented in CoCoALib.
\end{abstract}




\end{frontmatter}


\section{Introduction}

We present three distinct but related algorithms.  Cyclotomic
polynomials enjoy many special properties, and it is useful being able
to identify them quickly~---~our first algorithm does this, improving
upon the methods in~\cite{BD89}.  Our
second algorithm finds which cyclotomic polynomials appear as factors:
this refines the method of Beukers {\&} Smyth~\cite{SB02}.  The third
algorithm is an improved method
for detecting \emph{LRS-degenerate} polynomials compared to those
given in \cite{CDM11}: recall that a polynomial is LRS-degenerate
iff it has two distinct roots whose ratio is a root of unity.
In~\cite{CDM11} the property was named simply \emph{degenerate},
derived from the notion of \emph{degenerate Linear Recurrence Sequence (LRS).}

Interest in LRS-degenerate polynomials derives mostly from the
Skolem-Mahler-Lech theorem~\cite{LOW21}, which characterises the zeroes
of a LRS: in particular, it states that if the LRS has infinitely many zeroes then
the associated polynomial is LRS-degenerate.

All three of our algorithms are based on ``brute force applied
intelligently'' in contrast to the theoretical elegance of earlier
methods presented in \cite{BD89,SB02,CDM11}.  In particular our algorithm for detecting
LRS-degeneracy is a realization of the modular approach hinted at towards
the end of Section~3 in \cite{CDM11}.

Our emphasis is on practical utility, so we do not present asymptotic complexity
analyses of the algorithms, since they are of little utility in this
context (and would detract from the focus of this paper).  We do
present timings; in particular, we show that our LRS-degeneracy
algorithm is significantly faster than the two which were presented in
\cite{CDM11}.

\subsubsection*{Acknowledgements}

Abbott began working on this topic while at the University of Passau; he
then transferred to Rheinland-Pf\"alzische Technische Universit\"at in
Kaiserslautern, where he is supported by the Deutsche
Forschungsgemeinschaft, specifically via Project-ID~286237555~--~TRR 195.
Mexis is working as part of the Project
``CySeReS-KMU: Cyber Security and Resilience in Supply Chains with focus on SMEs''
(project number BA0100016), co-funded by the European Union through
INTERREG VI-A Bayern-Österreich 2021--2027.

The algorithms presented here have been implemented in C++ as part of CoCoALib~\cite{CoCoALib}, from version~0.99856.
The timings reported in Sections~\ref{sec:TimingsCyclotomicTest}, \ref{sec:TimingsCyclotomicFactors}
and \ref{sec:TimingsLRSDegenerate} were obtained on an AMD Ryzen~9 7900X processor, and are given in seconds.

\clearpage

\section{Notation, Terminology, Preliminaries}

Here we introduce the notation and terminology we shall use.

We shall often restrict attention to square-free, content-free polynomials
in $\ZZ[x]$: a general non-zero polynomial in $\QQ[x]$ may readily be rescaled
by clearing denominators and dividing out the content to yield a content-free
polynomial in $\ZZ[x]$.  We can
easily compute the \textbf{radical} (\ie~product of the factors in a
square-free factorization) as $f/\gcd(f,f')$ where $f'$ denotes the
formal derivative of $f$. We shall write the \textbf{reverse} of
a non-zero polynomial as $\rev(f) = x^d f(x^{-1})$ where $d = \deg(f)$.

\begin{definition}
\label{def:prim-root-unity}
Let $K$ be a field, and let $\zeta \in K$.  We say that $\zeta$ is a \textbf{root of unity}
if there is $n \in \NN_{>0}$ such that $\zeta^n=1$.  Additionally,
we say that $\zeta$ is a \textbf{primitive $\bm{k}$-th root of unity} if $\zeta^k=1$ and
there is no positive $\kappa < k$ with $\zeta^\kappa = 1$.
\end{definition}

\begin{example}
\label{ex:PrimKthRoot}
  Let $k \in \NN_{\ge 2}$, and let $p = 1+ks$ be prime (for some $s \in \NN$).
  Let $g$ be a primitive root modulo $p$, then $\zeta = g^{(p-1)/k} = g^s$ is a primitive
  $k$-th root of unity in the finite field $\FF_p$.
\end{example}

We recall below some fairly standard notation.
\begin{definition}
\label{def:height}
Let $f = \sum_{j=0}^d a_j x^j\in \CC[x]$.  We define the \textbf{height}
of $f$ to be $H(f) = \max \{ |a_0|, |a_1|, \ldots, |a_d| \}$ with
the special case $H(0) = 0$.
\end{definition}

\begin{definition}
\label{def:totient}
Let $\totient$ denote \textbf{Euler's totient function} for positive
integers: namely for every prime $p$ we have $\totient(p^r) = p^{r-1}
(p-1)$, and if $m,n$ are coprime integers then $\totient(mn) =
\totient(m) \, \totient(n)$.  In particular, $\totient(1) = 1$.
\end{definition}

\begin{definition}
\label{def:moebius}
Let $\MoebiusFn$ denote the \textbf{Möbius function} for positive
integers: namely, if $n$ is square-free with $r$ prime factors then
$\MoebiusFn(n) = (-1)^r$, otherwise $\MoebiusFn(n) = 0$.  In particular, $\MoebiusFn(1) = (-1)^0 = 1$.
\end{definition}

\begin{definition}
\label{def:cyclo}
Let $k \in \NN_{\ge 1}$.  We write $\Phi_k$ to denote the \textbf{$\bm{k}$-th cyclotomic
polynomial} (or \textbf{cyclotomic polynomial of index $\bm{k}$}), \ie~the unique monic polynomial of smallest degree whose roots are all primitive $k$-th roots of unity.
It is an elementary result that $\Phi_k(x) \in \ZZ[x]$ and $\deg(\Phi_k) = \totient(k)$.
\end{definition}

\clearpage 
\begin{lemma}
  \label{lem:cyclo-prods}
  We recall two product relations for cyclotomic polynomials:
  \begin{itemize}
    \item[(a)] The classical product
\[
x^k-1 \;=\; \prod_{d \mid k} \Phi_d(x)
\]
\item [(b)] Applying M\"obius inversion to~(a) gives
\[
\Phi_k(x) \;=\; \prod_{d \mid k} (x^d-1)^{\MoebiusFn(k/d)}
\]
  \end{itemize}
\end{lemma}

\begin{definition}
For a positive integer $k$ we write $G_k$ to denote the $\bm{k}$\textbf{-th Graeffe transformation}~\cite{GraeffeWiki} which maps roots to their $k$-th
powers.  Specifically we define $G_k(f(x)) = \res_y(f(y),\, x-y^k)$.
\end{definition}

\begin{remark}
The Graeffe transform exists in the same ring as the original polynomial.
Also for $k=2$ and $k=3$ the Graeffe transform can be computed quickly and
easily: \eg~$G_2(f) = (-1)^d ( f_{even}^2 - x\, f_{odd}^2)$ where $d = \deg(f)$, and
$f_{even}$ and $f_{odd}$ are defined by $f(x) = f_{even}(x^2) + x\, f_{odd}(x^2)$.

For composite $k$ we can compute a succession of simpler Graeffe transforms based on $G_{ab}(f) = G_a(G_b(f))$.

If several Graeffe transforms for consecutive $k$ are needed,
one can use one of the approaches mentioned in \cite{H63};
we implemented approach~4 from that work, based on Newton symmetric functions,
due to it being iteratively applicable to consecutive orders $k$.
\end{remark}

\begin{definition}
\label{def:kLRSd}
Let $k \in \NN_{\ge 2}$ and non-zero $f \in \CC[x]$.  Then we say that $f$ is $\bm{k}$\textbf{-LRS-degenerate}
iff $f$ has two distinct roots $\alpha,\beta \in \CC$ such that
$\alpha/\beta$ is a primitive $k$-th root of unity.
We also say simply that $f$ is \textbf{LRS-degenerate} to mean that there
exists at least one order $k \in \NN_{\ge 2}$ such that $f$ is $k$-LRS-degenerate. We are primarily interested in the case $f \in \ZZ[x]$.
\end{definition}

\begin{definition}
\label{def:LRSequiv}
Let non-zero $f,g \in \CC[x]$.  Then we say that $f$ and $g$ are \textbf{LRS-degeneracy equivalent}
iff  there are scale factors $\lambda,\mu \in \CC_{\neq 0}$ such that $f(x) = \mu\, x^d\, g(\lambda x)\;$ or
$\;f(x) = \mu\, x^d \, g(\lambda/x)$ for some $d \in \ZZ$.
It is elementary to show that this is an equivalence relation.
\end{definition}

\begin{remark}
If $f$ and $g$ are LRS-degeneracy equivalent then $f$ is $k$-LRS-degenerate iff $g$ is $k$-LRS-degenerate.
\end{remark}

\section{Algorithms for Testing Cyclotomicity}
\label{sec:test-cyclo}

An elegant method for determining whether a polynomial is cyclotomic
was presented in~\cite{BD89}.
We present a computational approach which is more effective than elegant;
it assumes that the polynomial is represented explicitly, and that we can
easily obtain the coefficients of the terms of highest (or lowest) degree.

Let $f \in \ZZ[x]$ be a non-constant, content-free polynomial.  We explain how
we check whether $f$ is cyclotomic; and if it is indeed cyclotomic,
we determine its index.

We start with a few trivial ``quick checks'' which every
cyclotomic polynomial will pass (but also some other polynomials may
pass too).  We handle specially the case where $f$ is of the form
$g(x^r)$ for some exponent $r>1$: this permits us to concentrate on
identifying only cyclotomic polynomials with square-free index.
If all the quick checks pass, we apply a more costly absolute test.

\subsection{Preliminary quick checks}
\label{sec:preliminary-checks}

The first 3 checks are very simple and quick, and should be applied
in the order presented.

\begin{itemize}
\item[Q1] If $f$ is not monic, it is not cyclotomic.  We henceforth assume
that $f$ is monic.

\item[Q2] If $\deg(f)$ is odd, it is not cyclotomic, \textbf{except for} the
  following two cases: $f(x) = x-1 = \Phi_1(x)$ or $f(x) = x+1 = \Phi_2(x)$.
  We henceforth assume that $\deg(f)$ is even.

\item[Q3] If we can obtain the constant coefficient of $f$ quickly, we can check
that this is $1$~---~otherwise $f$ is not cyclotomic.
\end{itemize}

We defer testing palindromicity until later: see check~Q5a in Section~\ref{sec:further-checks}.

\subsubsection{Coefficient of \texorpdfstring{$x^{d-1}$}{x\^{}(d-1)} and the special case \texorpdfstring{$f(x) = g(x^r)$}{f(x) = g(x\^{}r)}}

\begin{itemize}
\item[Q4a] Check the coefficient of $x^{d-1}$ where $d = \deg(f)$.
If it is not $-1$, $0$ or $1$ then $f$ is not cyclotomic.  If it is $0$, we take special action (described immediately below: Q4b--Q4e).  If it is non-zero, we skip steps Q4b--Q4e.

\item[Q4b] We check two easy special cases.  If $f(x) = x^d{-}1$ with $d > 1$ then it is not
cyclotomic.  If $f(x) = x^d{+}1$ with $d > 1$ then it is cyclotomic iff
$d = 2^s$ in which case the index is $2d$.

\item [Q4c] Define $r$ by saying that the
second highest term has degree $d{-}r$.  If the coefficient of $x^{d-r}$
is not $-1$ or $1$ then $f$ is not cyclotomic; and if $r \notdiv d$ or
any term in $f$ has degree not a multiple of $r$ then $f$ is not
cyclotomic.  Otherwise we have the largest exponent $r$ such that
$f(x) = g(x^r)$.

\item[Q4d] If $\frac{d}{r}$ is odd or $\totient(\rad(r)) \notdiv \frac{d}{r}$ then
  $f$ is not cyclotomic~---~here $\rad(r)$ denotes the \emph{radical} of $r$,
  \ie~the product of all distinct primes dividing~$r$.

\item[Q4e] We conduct the ``square-free index'' test of Section~\ref{subsec:PrefixMethod} on
  $g(x) = f(x^{1/r})$.  If $g(x)$ is not cyclotomic then neither is $f(x)$.
  If however $g(x)$ is cyclotomic with index $k$, we then check whether $\rad(r) \mid k$.
  If so then $f(x)$ is cyclotomic with index $rk$, otherwise $f$ is not cyclotomic.
\end{itemize}

\begin{example}
  Let $f = x^4+x^2+1$.  So $d=4$, and in Q4a we see that $x^{4-1}$ has coefficient 0. Check~Q4b does not reject $f$.
  In Q4c we obtain $r=4-2=2$, and find that $f(x) = g(x^2)$
  for $g = x^2+x+1$. Check~Q4d does not reject $f$.  In Q4e we identify $g$ as $\Phi_k(x)$ with index $k=3$.  We compute $\rad(r) = 2$,
  and see that $\rad(r) \notdiv k$; thus $f$ is not cyclotomic.  Indeed, $f = \Phi_3 \Phi_6$.

  In contrast, if instead we work with $f=x^6+x^3+1$, we obtain $r=3$ in Q4c, and
  as before we determine that $g = \Phi_k$ with $k=3$; but this time $\rad(r) \mid k$; so we conclude in Q4e that $f = \Phi_9$.
\end{example}

\subsection{Further quick checks}
\label{sec:further-checks}

We present some further simple checks; unlike the checks Q1--Q4b,
these checks can access all the coefficients of the polynomial,
so they are typically costlier.  Checks~Q5c and~Q5d are valid only for
cyclotomic polynomials with square-free index, so they must be conducted
only after all checks in Section~\ref{sec:preliminary-checks}
have been made; in contrast, checks~ Q5a and~Q5b are independent of the other checks.


\begin{itemize}
\item[Q5a] Check whether $f$ is palindromic.  If not then $f$ is not cyclotomic.

\item[Q5b] Check that the coefficient height of $f$, denoted~$\CoeffHeight(f)$, is within bounds: here we can combine
information from sequence \texttt{A138474} at OEIS~\cite{OEIS}, and from a table of
height bounds based on degree (see also Section~\ref{subsec:HeightBounds}).

\item[Q5c] Compute $V_1 = f(1)$; this is just the sum of the coefficients.
If $V_1 \neq 1$ and $V_1 \neq d{+}1$ then $f$ is not cyclotomic.
If $V_1 = d{+}1$: check that $d{+}1$ is prime and that
$\CoeffHeight(f)=1$; if so, $f$ is cyclotomic with (prime) index $d{+}1$;
otherwise $f$ is not cyclotomic.

\item[Q5d] Compute $V_{-1} = f(-1)$; this is an ``alternating sum'' of
the coefficients.
If $V_{-1} \neq 1$ and $V_{-1} \neq d{+}1$ then $f$ is not cyclotomic.
If $V_{-1} = d{+}1$: check that $d{+}1$ is prime and that
$\CoeffHeight(f)=1$; if so, $f$ is cyclotomic with index $2(d{+}1)$;
otherwise $f$ is not cyclotomic.
\end{itemize}

Observe that all four of these checks can be performed together
in a single scan over the coefficients.  Also, checks~Q5c and~Q5d can be
seen as a special case of our method for identifying cyclotomic factors
presented in Section~\ref{sec:cyclo-factors}.

\begin{remark}
We implemented these checks in CoCoALib, but found that they are not as fast
as the ``matching prefix'' method described in Section~\ref{subsec:PrefixMethod}.
This is partly explained by the fact that in CoCoALib ``reading'' all the
coefficients of a polynomial is costly.  Thus our final implementation does
not actually employ these checks.
\end{remark}

\vspace*{-15pt} 
\subsection{Testing for a cyclotomic polynomial with square-free index}
\label{subsec:PrefixMethod}
\vspace*{-7pt} 

Here we assume that all the checks from Section~\ref{sec:preliminary-checks} have
been made, so that if $f$ is cyclotomic then its index must
be square-free (and at least~3).  As already indicated, the test we
give here involves ``smart brute force'', and is decidedly more
computationally costly than the quick checks in Section~\ref{sec:preliminary-checks}.

The idea here is simple: we generate a list of possible candidate
indexes, and then repeatedly whittle it down by computing prefixes of
the corresponding cyclotomics, and keeping only those indexes whose
prefixes match the highest terms in $f$.  If more than one candidate
index survives, we refine the list again using longer prefixes.
In Algorithm~\ref{algm:CycloIndex}~(\textbf{CycloIndex}) we write $f_m$ to mean the $m$-prefix of $f$,
\ie~all terms of degree greater than $d{-}m$.

\begin{remark}
  In Algorithm~\ref{algm:CycloIndex}~(\textbf{CycloIndex}) in Steps~\ref{step:compute-cyc}
  and~\ref{step:final-check}, it suffices to verify
the $\frac{d}{2}$-prefix and then check that $f$ is palindromic.
\end{remark}

\begin{algorithm}[!ht]
\caption{\textbf{(CycloIndex)}}
\label{algm:CycloIndex}
\begin{algorithmic}[1]
    \Require $f \in \ZZ[x]$ monic, even degree $d$, and $a_{d-1}$, coefficient of $x^{d-1}$, is $\pm 1$
    \Ensure A positive integer $k$ if $f$ is the $k$-th cyclotomic polynomial, otherwise \emph{not cyclotomic}
    \Statex
    \State $L \gets \totient^{-1}(d)$ keeping only the square-free values
    \State Filter $L$: keep only those $n \in L$ for which $\MoebiusFn(n) = -a_{d-1}$
    \State Choose initial prefix length $m=32$
    \While{$L$ contains more than 1 element}
    \State $L_{new} \gets [\,]$, an empty list
    \ForEach{$n \in L$}
    \State Compute $p$, the $m$-prefix of $\Phi_n$
    \IIf{$p = f_m$} append $n$ to $L_{new}$ \EndIIf \label{step:put-n}
    \EndForEach
    \State $L \gets L_{new}$
    \State $m \gets 4m$ \label{step:incr4}
    \EndWhile
    \IIf{$L=[\,]$} \Return \emph{not cyclotomic} \EndIIf
    \State $k \gets L[1]$, compute $\Phi_k$    \label{step:compute-cyc}
    \IIf{$f = \Phi_k$} \Return $k$ \ElseI \Return \emph{not cyclotomic} \EndIIf \label{step:final-check}
\end{algorithmic}
\end{algorithm}

\begin{remark}
In Algorithm~\ref{algm:CycloIndex}~(\textbf{CycloIndex}) the initial value for the prefix
length $m$, and the increment in Step~\ref{step:incr4} were chosen because they
worked well in our tests.  We did not try ``optimizing'' the strategy
as it seemed to be unnecessary.  After the first iteration of the main
loop, $L$ typically contains only very few elements.
\end{remark}

\begin{remark}
  A referee described a potentially more efficient version of Algorithm~\ref{algm:CycloIndex}~(\textbf{CycloIndex}):
  an outer loop iterates over the candidate indexes, and an inner loop checks each
  candidate by constructing successively longer prefixes until the candidate is
  either excluded or confirmed as sole candidate for final verification. When the $m$-prefix of the candidate $k$ matches, and no other candidate $k'$ satisfies $\gcd(k', (m{-}1)\#) = \gcd(k, (m{-}1)\#)$ then $k$ is a sole candidate; here $(m{-}1)\#$ denotes the primorial of $m{-}1$.
  The referee's approach may indeed be faster if the input polynomial is actually cyclotomic, and its
  index appears early in the list of candidates.  Nevertheless, we expect the difference in overall speed to be
  minor. We point out also that the final verification represents typically the main part of the total computational cost.
  Both Algorithm~\ref{algm:CycloIndex}~(\textbf{CycloIndex}) and the referee's alternative can be adapted to forgo the final verification, for a faster but unverified result; even then, we doubt that
  there would be any significant difference in speed.
\end{remark}

\subsubsection{Computing an \texorpdfstring{$m$}{m}-prefix}
\label{subsec:ComputePrefix}

Let $n \in \ZZ_{>2}$ be square-free.  Since cyclotomic polynomials are
palindromic, prefixes and suffixes are equivalent.  It is convenient
here to discuss suffixes: namely $\Phi_n(x) \bmod x^m$.  We can compute
an $m$-suffix cheaply via the Möbius inverted product:
\[
\Phi_n(x) \;=\; \prod_{d \mid n} (x^d-1)^{\MoebiusFn(n/d)} \;=\; \prod_{d \mid n} (1-x^d)^{\MoebiusFn(n/d)}
\]
We may negate the factors because $n$ has an even number of divisors.
Observe that multiplying by $1-x^d$ is just a \emph{shift-and-subtract} operation, and
indeed dividing by $1-x^d$ can also be implemented as a \emph{shift-and-add} operation
(working up from lowest to highest degree).  We may ignore all $d \ge m$.

\subsubsection{Coefficient Bounds and Machine Integers}

In Algorithm~\ref{algm:CycloIndex}~(\textbf{CycloIndex}) at
Step~\ref{step:put-n} we use $f_m$, the $m$-prefix of $f$.  We incorporate
a coefficient height check while generating $f_m$ from $f$; if a
coefficient is too large then we know that $f$ cannot be cyclotomic.
Furthermore, so long as $m$ is not too large, all arithmetic can be done with
machine integers: the outermost coefficients of a cyclotomic
polynomial cannot be too large.  A table of bounds for each of
the outermost coefficients is available as
sequence \texttt{A138474} at OEIS~\cite{OEIS}.  On a 32-bit computer
overflow cannot occur with $m \le 808$; and on a 64-bit computer all
$m \le 1\,000$ are safe~---~the table at OEIS stops at index $1\,000$; the
true limit for 64-bit computers is likely considerably higher.

For prefixes of length $m > 1\,000$ we can use
the tabulated values of height bounds for cyclotomic polynomials (up to index
$2\,583\,303\,555$~---~see Section~\ref{subsec:HeightBounds}).
Additionally, for each candidate index we can apply the formulas below
to obtain potentially better bounds (especially when the index has few
prime factors).
Our implementation gives an error if the best bound found exceeds the
largest machine integer.

In~\cite{B12} there are formulas bounding the height of $\Phi_n$ and of
its power-series inverse $1/\Phi_n(x)$ for an odd, square-free index
$n = p_1 p_2 \cdots p_k$ where $p_1 < p_2 < \cdots < p_k$.
Interestingly, all the bounds are independent of the largest two prime factors.
If $k<3$ then the height bound is~$1$.
If $k=3$ then the bound is $\tfrac{2}{3} p_1$ according to preprint~\cite{JMRSV23},
and on the basis of this result the article~\cite{B12} establishes bounds on the height
of $\Phi_n$ for $k=4,5,6$:
\[
\frac{2}{3} p_1^3 p_2 \qquad \frac{2}{9} p_1^7 p_2^3 p_3 \qquad \frac{32}{729} p_1^{15} p_2^7 p_3^3 p_4
\]
The bounds on the height of the inverse have a similar form but with different
constant factors.  If the proof in~\cite{JMRSV23} is not confirmed as correct then~\cite{B12} obtains other constant factors.
A fully general formula is given in~\cite{BPV81}.  So far we have implemented
just the simple cases $k=1$ and $k=2$.

If some of the prime factors of $n$ are greater than the prefix length $m$,
we may sometimes obtain lower bounds.  For instance, if $n = p_1 p_2 p_3$\ and $p_1 < m \le p_2$ then the only
non-trivial truncated factors in the M\"obius product are the
same as those required to compute the $m$-prefix of $\Phi_{p_1}$.  Thus the
height bound for index $p_1$ limits the~$m$ outermost coefficients also
of $\Phi_n$; indeed, in this instance we know that the coefficients
of $x^e$ for $p_1 \le e < p_2$ must be zero.
Depending on whether the number of prime factors bigger than~$m$ is even
or odd, we use the height bound for the cyclotomic polynomial or for its inverse.

\subsubsection{Computing preimages under \texorpdfstring{$\totient$}{phi}}


All elements of $\totient^{-1}(d)$ can be computed quickly enough by a simple recursive ``tree search''~\cite{CCS06}.
The algorithm may easily be adapted to produce only the square-free preimages.
The algorithm could also produce the preimages naturally in factored form, which is ideal for our
application~---~however we have not yet implemented this ``optimization''.

\subsection{Testing via evaluation}
\label{subsec:EvalMethod}

We recall the quick checks~Q5c and Q5d from Section~\ref{sec:further-checks}:
these perform evaluation at~$1$ and at~$-1$ respectively, combined with a check
that the resulting value is permissible.  We now show that evaluation at~$2$ provides
us with a simple way of obtaining a unique candidate index (or proof that the polynomial
is not cyclotomic).  However, our implementation of the prefix method of
Algorithm~\ref{algm:CycloIndex}~(\textbf{CycloIndex}) turned out to be faster.

\subsubsection{Obtaining a candidate index (via remainders of powers)}

An irreducible, monic polynomial is cyclotomic iff it divides $x^k-1$ for some
exponent $k > 0$.  The smallest such $k$ is equal to the index
of the cyclotomic polynomial.  We could use this as an absolute
test for cyclotomicity: check whether $f(x)$ divides $x^k-1$
for $k=1,2,3, \ldots$  Bounds on the inverse totient function
ensure that only finitely many values of $k$ need to be checked.
This approach was presented in~\cite{BD89}.
However, actually doing this is rather inefficient.

We could make the check faster by computing the canonical remainder of
$x^r \bmod f(x)$ for $r=2d,\, 3d,\, 4d,\, \ldots$.  Bounds on the inverse totient function show that for degrees
$<36{\,}495{\,}360$ (see sequence \texttt{A355667} at OEIS~\cite{OEIS}) we can stop at $r=6d$ since $k$ cannot
exceed this value.  If the polynomial is indeed $\Phi_k$ then once we
try the first exponent $r \ge k$ we obtain remainder $x^s$ (with $s=r-k$), and it
is easy to identify if the remainder has this form.  From this remainder
we easily obtain $k$ as $r - s$.  However again, actually
computing these polynomial remainders would still be rather costly.

Instead we mimic the above but under the mapping $x \mapsto 2$.
We evaluate $f(2)$, and check that $2^{d-1} < f(2) < 2^{d+1}$.
Then we compute $2^r \bmod f(2)$ for $r=2d,\, 3d,\, 4d,\, \ldots, 7d$;
to advance to the next value in the sequence we just multiply by $2^d$ then
reduce modulo $f(2)$.
Observe that on a binary computer it is easy to detect if the remainder is
of the form $2^s$, in which case we obtain the candidate index $k = r-s$.
We finally verify that $\deg(f) = \totient(k)$ and if so, also that
$f = \Phi_k$ by computing $\Phi_k$.

\begin{remark}
  We note that the final check (that $f = \Phi_k$) is necessary, at least for higher degrees and
  index $k$ not of the form $p$ or $2p$ for some prime~$p$.
  Consider the polynomial $\Phi_k(x) + x^{d/2-3}\, (x^2-1)^2\, (2x^2-5x+2)$ where $d = \deg(\Phi_k)$;
  it is palindromic, and will pass all the evaluation tests, since the added polynomial
  vanishes at $x = \pm1$ and $x=2$.  For higher degrees the coefficient height check will
  pass also.
\end{remark}

\subsubsection{Explicit Algorithm}
\label{sec:CycloIndexByEval}

The algorithm has to exclude just a few edge cases: namely, $\Phi_1$ and $\Phi_6$
which are exceptions in Zsigmondy's theorem.  It is convenient, but not necessary,
to exclude $\Phi_2$ as well.  We assume that the $\bmod$ function returns the least
non-negative remainder.

We note that on a binary computer it is quick and simple to recognize if an integer is
a power of~$2$ and to determine the corresponding exponent: we require this in Algorithm~\ref{algm:CycloIndexByEval}~(\textbf{CycloIndexByEval}) Step~\ref{step:power-of-two}.

\begin{algorithm}[!ht]
\caption{\textbf{(CycloIndexByEval)}}
\label{algm:CycloIndexByEval}
\begin{algorithmic}[1]
    \Require $f \in \ZZ[x]$ monic, even degree $d$, and $f \neq \Phi_6$
    \Ensure A positive integer $k$ if $f$ is \emph{very likely} the $k$-th cyclotomic polynomial, otherwise \emph{not cyclotomic}
    \Statex
    \State $v \gets f(2)$
    \IIf{$v \le 2^{d-1}$ or $v \ge 2^{d+1}$} \Return \emph{not cyclotomic} \EndIIf
    \State $p \gets 2^d \bmod v$; $a \gets p$
    \For{$r' \gets 2$  \To~$7$}   \Comment {where $r' \leftrightarrow \tfrac{r}{d}$; see also Remark~\ref{rmk:pwr2remainder}}
    \State $a \gets (ap) \bmod v$
    \IIf{$a = 2^s$ for some $s$} \Return $r'd-s$ \EndIIf            \label{step:power-of-two}
    \EndFor
    \State \Return \emph{not cyclotomic}
\end{algorithmic}
\end{algorithm}

\vspace*{-2pt} 
\subsubsection{Supporting arguments for evaluation at 2}

We justify Algorithm~\ref{algm:CycloIndexByEval}~(\textbf{CycloIndexByEval}) presented in Section~\ref{sec:CycloIndexByEval}.

\begin{lemma} \label{lem:CycloAtX=2Bounds}
Let $k \ge 2$, then  $2^{d-1} < \Phi_k(2) < 2^{d+1}$ where $d = \deg(\Phi_k)$.
\end{lemma}

\begin{proof}
  We just use Theorem~5 from~\cite{TV11} with $b=2$.  Since $\Phi_k(2)$ is odd, we can change the inequalities to strict ones.
\end{proof}

\begin{lemma} \label{lem:CycloAtX=2Divides}
Let $k \in \NN_{>0}$ with $k \neq 6$.  Let $r \in \NN_{>0}$ be
such that $\Phi_k(2) \mid 2^r{-}1$.  Then $k \mid r$.
\end{lemma}

\begin{proof}
We observe that for $r>0$ we have $2^r-1 \neq 0$.  Since $\Phi_k(x)$
divides $x^k-1$, we have that $\Phi_k(2)$ is non-zero and divides
$2^k-1$.  Moreover, since $x^k-1$ divides $x^{sk}-1$ for all $s \in
\NN_{>0}$, we immediately have that $\Phi_k(2)$ divides $2^{sk}-1$ for
all $s \in \NN_{>0}$.

To show that there are no other exponents $r$ we invoke Zsigmondy's
Theorem (see Section~\ref{sec:ZsigmondysTheorem}) with $a=2$ and $b=1$.
\end{proof}

\begin{remark} \label{rmk:pwr2remainder}
We explain here why the exponents we use for testing (namely, $2d, 3d, 4d, \ldots, 7d$) are sufficient.

Assuming the polynomial is $\Phi_k$ for some square-free index $k$, we
know the value of $\totient(k) = \deg(\Phi_k)$.  Successive maxima of
$n/\totient(n)$ arise when $n$ reaches a primorial number: the
least $n$ where $n > 7 \totient(n)$ is at about $n \approx 1.3 \times
10^{16}$ when $\totient(n) \approx 1.8 \times 10^{15}$.  A fully general
implementation would compute the greatest exponent to test based on the degree
of the polynomial being tested.

Now we explain why we can ``make jumps'' of size $d$.  For any $\delta \in \NN$ we have that
$2^{k+\delta} \equiv 2^\delta \pmod{\Phi_k(2)}$.  So if $\delta <
\deg(\Phi_k)$ then by Lemma~\ref{lem:CycloAtX=2Bounds} the least
non-negative remainder of $2^{k+\delta}$ modulo $\Phi_k(2)$
is $2^\delta$, \ie~a power of $2$: something we
can test quickly on the computer.

 Let $s \in \NN$ be such that $sd < k \le (s{+}1)d$.  So the least
 non-negative remainder of $2^{sd}$ modulo $\Phi_k(2)$ is not a power of
 two (by Lemma~\ref{lem:CycloAtX=2Divides}), but $(s{+}1)d =
 k+\delta$ for some $\delta$ with $0 \le \delta < d$.  Hence
 $2^{s(d+1)} \equiv 2^\delta \pmod{\Phi_k(2)}$, and the right hand side
 is the least non-negative remainder, which we can recognise as a
 power of 2.
\end{remark}

\subsubsection{Height bounds for cyclotomic polynomials}
\label{subsec:HeightBounds}

In Q5b we check the height of $f$.  The outermost coefficients can be
checked via \texttt{A138474} as in Section~\ref{subsec:ComputePrefix}, but once
that table is exhausted we need another way: we use a second table as
described here (up to degree $11\,612\,159$).  If the degree of $f$ is too
large for this second table, the height check is skipped.

Since there are only finitely many cyclotomic polynomials of each
degree, there is a well-defined maximum coefficient height (as a
function of degree).  We do not know a nice formula for this maximum,
but use a precomputed table up to some degree limit.  For compactness
we employ a non-decreasing table of height maxima up to the given
degree, though this does weaken the check.

In our implementation the table of maximum coefficient heights
is represented as a cascade of \texttt{if} statements.
Letting \texttt{d} denote the degree, the first few lines are:
{\small
\begin{verbatim}
      if (d < 48) return 1;  // Phi(105)
      if (d < 240) return 2; // Phi(385)
      if (d < 576) return 3; // Phi(1365)
      if (d < 768) return 4; // Phi(1785)
      if (d < 1280) return 5; // Phi(2805)
      if (d < 1440) return 6; // Phi(3135)
      if (d < 3840) return 7; // Phi(6545)
      if (d < 5760) return 9; // Phi(15015)
      if (d < 8640) return 23; // Phi(21945)
\end{verbatim}
}

Note that our table is not the same as sequence \texttt{A160340} in
OEIS~\cite{OEIS}: our table comprises successive maximum heights by
degree, whereas \texttt{A160340} lists successive maximum heights by
index.  The first few entries do coincide, but then they diverge:
\eg~at degree $5760$ due to $\Phi_{15015}$, and at many other points.

Later entries in our table were derived from~\cite{AM11} whose data is
now available as part of~\cite{AM08Wayback}
under \emph{Library of data on the heights and lengths of cyclotomic polynomials.}

\subsubsection{Obtaining a candidate index (via table-lookup)}

We mention here a quick and easy way of identifying the index of a
cyclotomic polynomial, provided it has square-free index and degree $\le 1000$.
We do not use this in our implementation because of its limited range.

\clearpage
Let $S$ be the set of cyclotomic polynomials of degree up to $1000$ and
having square-free index.  Let the prime $p=57737$.  Then the mapping
$S \to \FF_p$ sending $f \mapsto f(2) \bmod p$ is 1--1.  So we can
precompute a table of indexes: \eg~an array whose $r$-th entry is the
index $k$ such that $r = \Phi_k(2) \bmod p$.  We fill ``empty'' array
entries with a value which is not a valid index (\eg~0).

We can use the table as follows: if $\deg(f) \le 1000$, compute $r =
f(2) \bmod p$.  If the table entry corresponding to residue $r$ is
``empty'' then $f$ is surely not cyclotomic.  Otherwise, let $k$ be
the table entry, and verify that $\totient(k) = \deg(f)$ and that
$-\MoebiusFn(k)$ is equal to the coefficient of $x^{d-1}$; if so then
maybe check that $2^k \equiv 1 \pmod{f(2)}$, and finally compute
$\Phi_k$ then test whether $f = \Phi_k$.

Such a table look-up is quick, but has
the obvious disadvantages that the table must be precomputed,
and that the table works only over a limited degree range.

\subsection{Timings}
\label{sec:TimingsCyclotomicTest}

In Table~\ref{table:timings-for-CyclotomicIndex} we give some sample timings of our implementation of the ``prefix
method'' (from Section~\ref{subsec:PrefixMethod}) in CoCoALib.  We
restrict to the case of identifying the index of a cyclotomic
polynomial \emph{without verification} in Steps~\ref{step:compute-cyc} and~\ref{step:final-check}, because the verification
step involves computing the full polynomial, and this computation
dominates.  To present meaningful timings we consider only
``difficult'' cases where a large prefix length is needed, and there
are many preimages to consider: for these examples, the smaller factor
of each index is a lower bound on the prefix length required.

\begin{table}[!ht]
    \centering
    \caption{Timings for \texttt{CyclotomicIndex}}
    \renewcommand{\arraystretch}{1.1}
    \begin{tabular}{|S[table-format=9.0]|c|S[table-format=5.0]||S[table-format=1.3]|}
        \hline
        {Index}   & Factors           & {$\#\totient^{-1}(d)$}    & {Time}          \\
        \hline
        124525451 & $8641 \times 14411$ &  3354   & 0.125 \\
        120507533 & $4481 \times 26893$ &  6861   & 0.140 \\
        334482719 & $5279 \times 63361$ &  8145   & 0.359 \\
        399083849 & $3329 \times 119881$ & 10819  & 0.406 \\
        \hline
    \end{tabular}
  \label{table:timings-for-CyclotomicIndex}
\end{table}

\section{Algorithms for Detecting Cyclotomic Factors}
\label{sec:cyclo-factors}

In this section we present an effective method of determining which
cyclotomic polynomials, if any, divide a given polynomial.
Once again, our method is ``refined brute force''.

An obvious way to determine the cyclotomic factors would be first to
compute the irreducible factors, and then use the method of
Section~\ref{sec:test-cyclo} to test each factor for being
cyclotomic~---~this approach is practical only for moderate degrees.  The task of finding all cyclotomic factors has already been studied and elegantly solved by Beukers and Smyth: in~\cite{SB02}
there is a remarkably short and simple algorithm for finding the
product of the cyclotomic factors of a square-free polynomial.  However, they
did not address the matter of identifying each individual cyclotomic
factor.  We present a new effective algorithm which identifies quickly
a list of indexes of cyclotomic factors.  It is possible that the list
contains a few false positives, though their presence is ``unlikely''; the
candidate indexes can be tested by trial division.

Let $f \in \ZZ[x]$ be the polynomial whose cyclotomic factors we wish to find.
We list here some potential preprocessing steps~---~they are not required for
correctness, yet could make the overall computation faster (though in our implementation there was no benefit):
\begin{itemize}
\item Take the radical: $f \gets f/\gcd(f,f')$
\item Extract the palindromic part: $f \gets \gcd(f,\rev(f))$
\item Extract the product of the cyclotomic factors using the Beukers {\&} Smyth algorithm~---~this can be costly if $f$ has high degree
\end{itemize}

\subsection{Finding cyclotomic factors by evaluation}
\label{subsec:CycloFactorsByEval}

The basis for our approach is evaluation of the polynomial (at several
rational points), and Zsigmondy's theorem (see~\cite{ZsigWiki}).

\subsubsection{Zsigmondy's Theorem}
\label{sec:ZsigmondysTheorem}

Let $(s_n)_{n \in \NN_{>0}}$
be a sequence of integers: let $k \in \NN_{>0}$, then a \emph{primitive prime}
for $s_k$ is a prime $p$ such that $p \mid s_k$ and $p \notdiv s_j$ for all
$j < k$. \\ 
Zsigmondy's theorem~\cite{ZsigWiki} tells us
that if we have the sequence $s_n = a^n - b^n$ for coprime
integers $a > b > 0$ then for every index the element $s_k$ has
a primitive prime.  The only exceptions are:
\begin{itemize}
  \item if $a = b+1$ then for $k=1$ we have $s_k = a^1 - b^1 = 1$ which has no prime factors
  \item if $a=2$, $b=1$ then for $k=6$ there is no such prime factor
  \item if $a+b = 2^r$ for some $r$ then for $k=2$ there is no such prime factor
\end{itemize}
We shall refer to the primitive primes for such sequences as \textbf{Zsigmondy factors}.
The specific case $a=2, b=1$ was already solved by Bang's Theorem (see~\cite{B1886,B1886_cont}).

\subsubsection{Evaluation at a rational}

For compactness we introduce a variant of function call notation.  Let
$f \in \ZZ[x]$ be non-zero of degree $d$, and let $\beta = p/q \in
\QQ_{>0}$ with $p$ and $q$ positive and coprime; then we write
\[
f(\beta)_{num} \;=\; q^d f(\beta);
\]
and if $f$ is zero then we define $f(\beta)_{num} = 0$.
Thus for all $f \in \ZZ[x]$ and all $\beta \in \QQ_{>0}$ we have $f(\beta)_{num} \in \ZZ$.
Also if $\beta \in \ZZ$ then $f(\beta)_{num} = f(\beta)$.

\begin{lemma}
\label{lem:Zsigmondy-for-numerators}
Let $\beta = p/q \in \QQ_{>1}$ with $p$ and $q$ coprime and positive.
Let the sequence $\sigma_k = p^k-q^k$; applying Zsigmondy's
theorem we obtain a Zsigmondy factor $z_k$ for each index
$k \in \NN_{>0}$~---~with the noted exceptions.
Let the sequence $\tau_k = \Phi_k(\beta)_{num}$.  Then the $z_k$ are
Zsigmondy factors for the $\tau$ sequence.
\end{lemma}

\begin{proof}
By 
Lemma~\ref{lem:cyclo-prods}(a),
we have $\sigma_k = \prod_{\kappa \mid k} \tau_\kappa$ for every $k \in \NN_{>0}$;
thus $\tau_k \mid \sigma_k$.  Since each $z_k \mid \sigma_k$ and
$z_k \notdiv \sigma_\kappa$ for any $\kappa < k$, we have that $z_k \mid \tau_k$
and $z_k \notdiv \tau_\kappa$ for any $\kappa {<} k$.  So the $z_k$ are Zsigmondy factors
for the sequence $\tau_k$.
\end{proof}

\subsubsection{Algorithm}
\label{sec:algm-cyclo-factors}

We present our algorithm in two parts: the inner Algorithm~\ref{algm:RefineCandidates}
(called \textbf{RefineCandidates}) which refines a list of candidate indexes,
and the outer Algorithm~\ref{algm:FindCycloFactors} (called \textbf{FindCycloFactors}) that calls the inner one
repeatedly.

The inner ``refinement'' algorithm takes 3 inputs: the polynomial $f$,
the evaluation point $\beta \in \QQ_{>1}$, and a list of candidate
indexes.  It produces a refined list of candidate indexes, which might
be empty.

\begin{algorithm}[!ht]  
\caption{\textbf{(RefineCandidates)}}
\label{algm:RefineCandidates}
\begin{algorithmic}[1]
    \Require Polynomial $f$, evaluation point $\beta = p/q \in \QQ_{>1}$, and $L$ a list of candidate indexes (excluding $1$ and $2$)
    \Ensure The refined list $L_{new} \subseteq L$ of candidate indexes
    \Statex
    \State $N \gets \gcd\bigl(f(\beta)_{num},\, f(\beta^{-1})_{num}\bigr)$ \label{step:refine-init}
    \Statex \Comment{if $f$ is known to be palindromic: $N \gets f(\beta)_{num}$}
    \IIf{$N = 0$}  $f \gets f/(qx-p)$; \Goto{step:refine-init}\EndIIf
    \State $L_{new} \gets [\,]$, an empty list
    \ForEach{index $k \in L$}
    \IIf{$\totient(k) > \deg(f)$} \Continue{} to next $k$\EndIIf
    \State $g \gets \gcd(p^k-q^k, N)$
    \IIf{$g=1$} \Continue{} to next $k$ \EndIIf
    \IIf{$\Phi_k(\beta)_{num} \notdiv N$} \Continue{} to next $k$ \EndIIf \label{step:skip-cyclo}
    \State Append $k$ to $L_{new}$
    \State $N \gets N/\gcd(N,g^\infty)$  \Comment{See Remark~\ref{rem:saturation}}
    \EndForEach
    \State \Return $L_{new}$
\end{algorithmic}
\end{algorithm}

\begin{remark}
  \label{rem:saturation}
  In Algorithm~\ref{algm:RefineCandidates}~(\textbf{RefineCandidates}) at
  Step~10 we write $N/\gcd(N,g^\infty)$ to mean $\lim_{s \to \infty} N/\gcd(N,g^s)$;
  it is the largest factor of $N$ coprime to $g$.
\end{remark}


The outer algorithm takes just the polynomial $f$ as input; it could
also perform the preprocessing steps mentioned earlier. It checks
explicitly for the factors $\Phi_1$ and $\Phi_2$, so that calls to the
inner algorithm are not complicated by the exceptions in Zsigmondy's
theorem. It always
chooses~$2$ as the first evaluation point (for reasons of computational
speed); thereafter it picks random small rationals.  The randomness serves
to preclude the possibility of constructing a \emph{small} input which
\emph{consistently} causes the algorithm to perform poorly~---~this is standard practice in randomized algorithms.

\clearpage
\begin{algorithm}[!ht]  
\caption{\textbf{(FindCycloFactors)}}
\label{algm:FindCycloFactors}
\begin{algorithmic}[1]
    \Require Primitive, non-constant polynomial $f \in \ZZ[x]$
    \Ensure A list of indexes of cyclotomic factors of $f$, possibly containing a few false positives
    \Statex
    \State Optional: Apply preprocessing steps from beginning of Section~\ref{sec:cyclo-factors}
    \State $L_{12} \gets [\,]$. If $f(1) = 0$, append $1$ to $L_{12}$. 
    If $f(-1) = 0$, append $2$ to $L_{12}$. 
    \State $L \gets$ a list containing all $k \in \NN_{>2}$ for which $\totient(k) \le \deg(f)$ \label{step:tot-filter}
    \State $L \gets \textbf{RefineCandidates}(f,2,L)$ \Comment{see Algorithm~\ref{algm:RefineCandidates}}
    \Repeat
    \State Pick a random evaluation point $\beta \in \QQ_{>1}$
    \State $L_{new} \gets \textbf{RefineCandidates}(f,\beta,L)$ \Comment{see Algorithm~\ref{algm:RefineCandidates}}
    \IIf{$L_{new}=[\,]$ or $L_{new} = L$} \Return $L_{12} \cup L_{new}$ \EndIIf \label{step:list-stable}
    \State $L \gets L_{new}$, and perform another iteration \label{step:next-iter}
    \Until{forever}
\end{algorithmic}
\end{algorithm}

\begin{remark}
If the input polynomial is monic and of even degree, we can perform,
before Step~\ref{step:tot-filter}, an initial cyclotomicity test (from
Section~\ref{sec:test-cyclo}) as this is very quick.  The initial
refinement at $\beta=2$ was quite effective in our tests: \ie~only few
false positives survived.  As written, Algorithm~\ref{algm:FindCycloFactors}~(\textbf{FindCycloFactors})
has difficulty excluding the indexes $3,4$ and~$6$; so
in Step~\ref{step:list-stable}, the first time the condition $L_{new} = L$ is satisfied,
we perform one more iteration but using a $\beta$ value chosen so that
each of $\Phi_k(\beta)_{num}$ for $k=3,4,6$ has a prime factor larger
than $10\,000$, \eg~suitable $\beta$ values are $\frac{117}{98}$, $\frac{133}{18}$ or $\frac{169}{6}$.
\end{remark}

\begin{remark}
An advantage of evaluating at rationals rather than just at integers is that the
evaluations are usually smaller.  There are $\Theta(B^2)$ reduced rationals with numerator and
denominator bounded by $B$. Let $\beta$ be such a rational and $f\in\ZZ[x]$ of degree $d$,
then $f(\beta)_{num}$ is bounded by $(d{+}1) \, \CoeffHeight(f) \, B^d$.
In contrast, we have only about $2B$ integers
giving evaluations below this bound.
\end{remark}

\begin{remark}
Let $f = \sum_{j=0}^d a_j x^j$ be non-constant. 
Let $p$ be a prime which divides $a_0, a_1, \ldots, a_r$ and $a_d, a_{d-1}, \ldots, a_{d-s}$;
note we do not need to compute $p$, it suffices to verify that $\gcd(a_0, a_1, \ldots, a_r,\; a_d, a_{d-1}, \ldots, a_{d-s}) \neq 1$.
Suppose $\Phi_k \mid f$, then $\psi_p(\Phi_k) \mid x^{-(r{+}1)}\, \psi_p(f)$ where $\psi_p$ is reduction modulo $p$.
Then we obtain the better bound for $\totient(k) = \deg(\Phi_k) \le \deg(f) -(r{+}s{+}2)$ in
Step~\ref{step:tot-filter} of Algorithm~\ref{algm:FindCycloFactors}~(\textbf{FindCycloFactors}).  If $f$ is palindromic
then we have $r=s$.

\end{remark}

\begin{remark}
We use a simple technique for generating ``random rationals''.  We
regard the rationals in $\QQ_{>1}$ as being ordered lexicographically
via the mapping $p/q \mapsto (p,q)$ with $\gcd(p,q)=1$.  A new rational is
chosen by jumping forward along this progression by a random amount;
by always jumping forward, we avoid generating the same value twice.
\end{remark}

\subsection{Timings}
\label{sec:TimingsCyclotomicFactors}

We exhibit timings for the ``difficult'' case of many cyclotomic
factors, since usually the first iteration whittles the list down to
just those factors actually present, possibly with just very few false
positives.

Each test polynomial is a product of distinct cyclotomic polynomials whose indexes
come from a random subset of specified cardinality from the range $\{1,2,\ldots,R\}$.
Each test set contains 10 such products.
The table records the average degree, and the average time;
no false positives were observed during testing.

\begin{table}[!ht]
    \centering
    \caption{Timings for \texttt{CyclotomicFactorIndexes}}
    \renewcommand{\arraystretch}{1.1}
    \begin{tabular}{|S[table-format=4.0]|S[table-format=3.0]|S[table-format=5.0]|S[table-format=1.2]|}
        \hline
        {Index range}   & {Num Factors}   & {Avg degree}    & {Avg Time}          \\
        \hline
        500 & 100 &  15731   & 0.10 \\
        1000 & 50 &  15696   & 0.10 \\
        1000 & 100 & 31487   & 0.22 \\
        1000 & 200 & 61444   & 0.50 \\  
        \hline
    \end{tabular}
\end{table}

We did observe a few false positives when using specially constructed square-free,
palindromic polynomials such as $f = \prod_{k=2}^{201} (kx-1)(x-k)$ which have
large ``fixed divisors''; the \emph{fixed divisor} is defined to be
$\gcd \{ f(n) \mid n \in \ZZ\}$.  In all cases the false positives were
indexes of cyclotomic polynomials of low degree (\eg~$3$, $4$ and $6$).

\section{Algorithms for Testing LRS-degeneracy}
\label{sec:TestLRS}

In this section we show that many LRS-degenerate polynomials exist,
recall the algorithms presented in~\cite{CDM11}, and then present our
new algorithm together with some sample timings.

\subsection{Some properties of LRS-degenerate polynomials}

We first show that infinitely many LRS-degenerate polynomials exist.
Indeed, every member of the family of polynomials $\{ x^2-n \mid n \in \ZZ_{\neq 0}\}$
is 2-LRS-degenerate; moreover, infinitely many of them are irreducible.
In fact, all the polynomials in the family are LRS-equivalent to each other.

Directly from the definition, we can easily see that:
\begin{itemize}
\item If $f \in \ZZ[x]$ is of the form $g(x^k)$ for some $k \in \ZZ_{\ge 2}$
  but not of the form $a_d x^d$ then $f$ is $\kappa$-LRS-degenerate for
  every $\kappa \ge 2$ dividing $k$~---~we note that $f$ may also be
  LRS-degenerate for other orders too.

\item If $f \in \ZZ[x]$ is $k$-LRS-degenerate then so is any (non-zero) multiple of~$f$.

\item Let $f \in \ZZ[x]$ be divisible by $x^d$ for some exponent $d \in \NN$ then $f$ is $k$-LRS-degenerate iff
  $f/x^d$ is.
\end{itemize}
We observe that, for a given $f \in \ZZ[x]$ not of the form $a_d x^d$, it is computationally
quick and simple to determine the largest $d,k \in \NN$ such that $f(x) = x^d g(x^k)$.
If $d$ and $k$ are both even then $f$ is an
even polynomial (\ie~$f(x) = f(-x)$); and if $d$ is odd and $k$ even then $f$ is an odd
polynomial (\ie~$f(x) = -f(-x)$).  We see that in both cases they are $2$-LRS-degenerate;
of course, this is also immediate from the definitions of evenness and oddness.

There are many further examples.  We mention two further quadratic
polynomials: $x^2+3x+3$ is 6-LRS-degenerate, and $x^2+x+1 = \Phi_3(x)$
is 3-LRS-degenerate.  Lemma~\ref{lem:cyclo-LRSd} immediately below summarises the
LRS-degeneracy of cyclotomic polynomials.  We can generate more
LRS-degenerate polynomials as follows: let $f$ be LRS-degenerate
and $g$ be not of the form $a_d x^d$ then $\res_y(f(y),\, g(xy))$ is
LRS-degenerate since its roots are $\{\alpha/\beta \mid f(\alpha)=0,\; g(\beta)=0\}$;
indeed, if $f$ and $g$ are irreducible then the resultant is likely
irreducible too (provided $g$ remains irreducible over the extension
generated by a root of $f$).

\begin{lemma} (cyclotomics are LRS-degenerate)
\label{lem:cyclo-LRSd}
  \begin{itemize}
    \item[(a)]
      For odd $k \in \NN_{\ge 3}$ the polynomial $\Phi_k(x)$ is
      $\kappa$-LRS-degenerate for every factor $\kappa$ of $k$
  (excluding $\kappa=1$).

 \item [(b)] For even $k \in \NN_{\ge 4}$ the polynomial $\Phi_k(x)$ is $\kappa$-LRS-degenerate
  for every factor $\kappa$ of $k/2$ (excluding $\kappa=1$).
  \end{itemize}
\end{lemma}

\begin{proof}
There is a 1--1 correspondence between primitive $k$-th roots of unity and
$S_k = \{ n \in \NN \mid 1 \le n < k  \text{ and }  \gcd(n,k)=1 \}$.
Namely, the set of all primitive $k$-th roots is just $\{\zeta^n \mid n \in S_k\}$
where $\zeta$ is any primitive $k$-th root.

Let $\kappa > 1$ be a factor of $k$, and let $\lambda = k/\kappa$.
Then $\zeta^{j+\lambda} / \zeta^j$ is a primitive $\kappa$-th root of unity.
We show how to find $j$ such that both $j$ and $j+\lambda$ are coprime
to $k$, \ie~so that both $\zeta^j$ and $\zeta^{j+\lambda}$ are primitive
$k$-th roots of unity.

Let the factorization of $k$ be $\prod_{s=1}^m p_s^{e_s}$ where the primes
$p_s$ are distinct, and each $e_s > 0$.  We now construct a pair of suitable exponents $(i,j)$.


\textbf{Case $k$ is odd.}
For each index $s$, pick non-zero residues $r_s \bmod p_s$ such
that $r_s {+} \lambda \not\equiv 0 \pmod{p_s}$; note that all $p_s > 2$, so a suitable $r_s$ exists.
Using CRT construct $j \in \{1,2,\ldots,k-1\}$ such that $j \equiv r_s \pmod{p_s}$.
Then $\gcd(j,k) = 1$ and $\gcd(j{+}\lambda,k)=1$.
Thus $\zeta^j$ and $\zeta^{j+\lambda}$ are primitive $k$-th roots,
while $\zeta^\lambda = \zeta^{j+\lambda} / \zeta^j$ is a primitive
$\kappa$-th root.

\textbf{Case $k$ is even.}
Essentially the same argument works, but
we must ensure that $\lambda$ is even; equivalently we must
require that $\kappa$ divides $k/2$.
\end{proof}

\begin{example}
  \label{ex:not-LRSd-with-order-a-factor}
  The property that if $f$ is $k$-LRS-degenerate then it is $\kappa$-LRS-degenerate for
  all factors $\kappa > 1$ dividing $k$ does not hold in general.  For example,
  $x^4 + 2 x^2 + 4 x + 2$ is $k$-LRS-degenerate only for $k=8$;
and $x^6 +3x^5 +6x^4 +6x^3 +3$ is $k$-LRS-degenerate only for $k=18$.  

\end{example}

\vspace*{-17pt} 
\subsection{Preprocessing prior to testing LRS-degeneracy}
\label{subsec:LRS-Preprocessing}
\vspace*{-6pt} 

To simplify later discussions we shall assume that polynomials to be
tested for LRS-degeneracy have been preprocessed.
In~\cite{CDM11} they describe some natural, simple preprocessing steps
which we recall and expand here:
\begin{itemize}
  \item we may assume that $f(0) \neq 0$: just divide by appropriate power of $x$
  \item we may assume that $f$ is non-constant and content-free in $\ZZ[x]$  
  \item we may assume that $f$ is square-free: just replace $f \gets f/\gcd(f,f')$
    \item we may assume that $f$ is not of the form $g(x^r)$ for some $r > 1$:\\
      if $f(x) = g(x^r)$ then $f$ is $k$-LRS-degenerate for each
      factor $k>1$ of $r$; if additionally $g(x)$ is $\kappa$-LRS-degenerate
      then $f$ is $r \kappa$-LRS-degenerate; $f$ may also be $k$-LRS-degenerate
      for other factors $k \mid r \kappa$.

  \item Optionally, we may also simplify the coefficients using Algorithm~\ref{algm:ReduceCoefficients} (\textbf{ReduceCoefficients}) from
  Section~\ref{subsubsec:CoeffRedn}.
\end{itemize}

If we are interested in knowing only whether $f$ is LRS-degenerate,
without actually determining the orders, then we can apply
the following additional steps:
\vspace*{-8pt} 
\begin{itemize}
  \item we may assume that $\gcd(f(x), f(-x)) = 1$: otherwise $f$ is clearly $2$-LRS-degenerate
  \item we may assume that $f$ has no cyclotomic factors (with index $>2$): otherwise it is trivially LRS-degenerate~---~\eg~use the method of Section~\ref{sec:cyclo-factors}
  \item we may also divide out any linear factors from $f$. 
\end{itemize}

\vspace*{-15pt} 
\subsubsection{Preprocessing: coefficient reduction}
\label{subsubsec:CoeffRedn}
\vspace*{-8pt} 

We can use LRS-degeneracy equivalence to preprocess polynomials to be tested for
LRS-degeneracy: the idea is to look for $\lambda$ and $\mu$ values
which ``simplify'' the coefficients.  A fully general approach appears
to be potentially costly, probably entailing integer factorization.
However, we can make a quick search for suitable $\lambda, \mu$; if we
are lucky, we can simplify the polynomial.

Algorithm~\ref{algm:ReduceCoefficients}~(\textbf{ReduceCoefficients})
is a simple preprocessing method which produces an LRS-degeneracy
equivalent polynomial whose coefficient size may be smaller than that
of the input $f$.  We run this algorithm twice: apply once to $f$ to
obtain $f_1$, then apply again to $\rev(f_1)$ to obtain $f_2$.  The
fully simplified polynomial is then $\rev(f_2)$.

\begin{remark}
  If the input is $f = a_d x^d + a_0$, with $a_d \, a_0 \neq 0$, then we can directly simplify to $x^d+1$.
  As presented, the algorithm will fail to fully simplify $f$ if it is unable to
  factorize $a_0$ or $a_d$.
\end{remark}

\begin{example}
  \label{ex:preprocessing-coeff-reduction}
  Let $f = 16x^4 +80x^3 +300x^2 +1000x +3125$.  The first run of the algorithm
  gives $\lambda = 1/2$ and $f_1 = x^4 +10x^3 +75x^2 +500x +3125$.
  The second run, on input $\rev(f_1)$, gives $\lambda = 1/5$
  and $f_2 = 5x^4+4x^3+3x^2+2x+1$.  So the final simplified form of $f$ is
  $\rev(f_2) = x^4 +2x^3 +3x^2 +4x +5$ and the combined $\lambda$ is $5/2$.
\end{example}

\begin{algorithm}[!ht]  
\caption{\textbf{(ReduceCoefficients)}}
\label{algm:ReduceCoefficients}
\begin{algorithmic}[1]
    \Require Content-free polynomial $f = \sum_{j=0}^d a_j x^j \in \ZZ[x]$ with $d > 0$ and $a_0,a_d \neq 0$
    \Ensure An LRS-degeneracy equivalent polynomial with ``reduced'' coefficients
    \Statex
    \IIf{$f(x) = g(x^r)$ for some $r > 1$} apply this algorithm to $g(x)$ producing $\tilde{g}(x)$; \Return $\tilde{g}(x^r)$ \EndIIf
    \State $g \gets \gcd(a_1, a_2,\ldots,a_d)$  \Comment{$a_0$ is deliberately excluded}
    \IIf{$g = 1$} \Return $f$ \EndIIf
    \State $\lambda \gets 1$
    \ForEach{``small'' prime factor $p$ of $g$}
    \State $m \gets \min\{\lfloor \mu_p(a_j)/j \rfloor \mid j=1,2,\ldots,d\}$
    \Statex \Comment{$\mu_p(n)$ is the multiplicity of $p$ in $n$; note that $m \in \ZZ$}
    \State $\lambda \gets \lambda / p^m$
    \EndForEach
    \State \Return $f(\lambda x)$ \Comment{element of $\ZZ[x]$}
\end{algorithmic}
\end{algorithm}

\vspace*{-22pt} 
\subsection{Factorization and LRS-degeneracy}
\vspace*{-10pt} 

It is clear from the definition that when testing a polynomial $f$ for
LRS-degeneracy we can work with the radical, $\rad(f)$, which may have
lower degree.  However we cannot, in general, use a finer
factorization.  We present an easy lemma, and then several examples
which show that we cannot determine whether $f$ is LRS-degenerate by
working independently on its irreducible factors.

\begin{lemma}
  Let $g(x) \in \CC[x]$ be not LRS-degenerate then $h(x) = g(x) \, g(-x)$ is
  $k$-LRS-degenerate only for $k=2$.
\end{lemma}

\begin{proof}
We show the contrapositive.

Suppose that $h(x)$ is $k$-LRS-degenerate for some $k \neq 2$.
Then $h(x)$ has roots $\alpha,\beta \in \CC$ such that $\alpha/\beta = \zeta_k$.  WLOG $g(\alpha) = 0$.

  If $g(\beta)=0$ then $g$ is $k$-LRS-degenerate.  Otherwise we have $g(-\beta) = 0$;
  so $g(x)$ has a root pair with ratio $-\alpha/\beta = -\zeta_k$,
  and hence $g(x)$ is LRS-degenerate (with order either $k/2$ or $2k$).
\end{proof}

\begin{remark}
It is easy to see that  if $f(x)\, g(x)$ is LRS-degenerate then so is $f(x)\, g(-x)$.
Note that $f(-x)\, g(x)$ is LRS-degeneracy equivalent to $f(x)\, g(-x)$,
and $f(-x)\, g(-x)$ is LRS-degeneracy equivalent to $f(x)\, g(x)$.
\end{remark}

\begin{example}
  Here we exhibit some products of non-LRS-degenerate polynomials which are $k$-LRS-degenerate for $k \neq 2$.
  Let $f_1 = 5x^2 +6x +5$ and $f_2 = 5x^2 +8x +5$.  Then neither $f_1$
  nor $f_2$ is LRS-degenerate, but the product $f_1 f_2$ is
  $4$-LRS-degenerate.

  Let $g_1 = 7x^2+2x+7$, $g_2 = 7x^2+11x+7$ and $g_3 = 7x^2+13x+7$.
  None of these polynomials is LRS-degenerate, but the product of any pair
  is:  $g_1 g_2$ is $3$-LRS-degenerate, while $g_1 g_3$ and $g_2 g_3$
  are $6$-LRS-degenerate.

  We can generate more examples of degree 2 by applying Graeffe
  transforms: \eg~$G_5(f_1)$ and $G_5(f_2)$ are not LRS-degenerate but
  their product is 4-LRS-degenerate.  This works for any $G_k$ with
  $k$ coprime to 4. Analogously, if we use the pair $g_1,g_2$ then we
  can apply $G_k$ with $k$ coprime to 3; for the pairs $g_1,g_3$ and
  $g_2,g_3$ we need $k$ coprime to $6$.

  We can generate similar examples with higher degrees, for instance:
  let $h$ be an irreducible polynomial which is not LRS-degenerate, and
  compute $h_i =  \res_y(f_i(y), h(xy))$ for $i=1,2$.  Most likely $h_1$ and $h_2$ are
  irreducible and not LRS-degenerate, but their product is $4$-LRS-degenerate.
\end{example}

\begin{example}
We can generate more pairs of non-LRS-degenerate polynomials whose
product is LRS-degenerate as follows.  Pick an index $k \ge 7$, and a
random polynomial $g \in \ZZ[x]$ with $\deg(g) < \totient(k)$.
Compute $m(y) = \res_x(\Phi_k(x),\, y-g(x))$; if this is reducible,
pick a different $g$.  Let $S$ be the set of irreducible factors of
$\res_y(m(y), \Phi_k(xy))$; then in many cases the factors are not
LRS-degenerate, but the product of any pair is $\kappa$-LRS-degenerate
for some $\kappa > 1$ dividing $k$.
\end{example}


\subsection{Recalling CDM Algorithm 1}
\label{algm:CDM1}

We recall briefly the two algorithms for detecting LRS-degenerate
polynomials which were published in~\cite{CDM11}.  The
presentation in that paper was mostly concerned with determining solely whether
a given polynomial is LRS-degenerate, rather than determining the
LRS-degeneracy orders.

These algorithms assume that the input $f \in \ZZ[x]$ has been
preprocessed (see Section~\ref{subsec:LRS-Preprocessing}) so that it
has degree at least $2$, and is content-free, square-free and $f(0) \neq
0$.

The first algorithm in~\cite{CDM11} computes $R_f =
\res_y(f(y), f(xy))/(x-1)^d$ where $d = \deg(f)$.  So $R_f$ is a
palindromic polynomial of degree $d^2-d$, since we assume that $f$ is
square-free.  It then determines whether $R_f$ has any cyclotomic
factors (with index $\ge 2$), since $f$ is $k$-LRS-degenerate
iff $\Phi_k(x)$ divides $R_f$.  In~practice, they
observe that this method becomes rather slow for $\deg(f) > 25$
because the resultant computation is costly.  Also they sought
cyclotomic factors by first computing a factorization into
irreducibles, which can be slow in some cases.  Our Algorithm~\ref{algm:FindCycloFactors}~(\textbf{FindCycloFactors})
or that from~\cite{SB02} would make the detection of cyclotomic
factors much faster, but would not reduce the cost of computing the resultant.
Our Algorithm~\ref{algm:FindCycloFactors}~(\textbf{FindCycloFactors}) produces also the orders of LRS-degeneracy.

\subsection{Recalling CDM Algorithm 2}
\label{algm:CDM2}

The second algorithm in~\cite{CDM11} replaces the costly resultant computation
by a succession of simpler resultant computations.  Effectively it searches for the lowest
order of LRS-degeneracy.  They report that it is usefully faster than their first algorithm if the polynomial
is indeed $k$-LRS-degenerate for some small $k$~---~but if it is not LRS-degenerate then
their first algorithm is often faster.

Recall that we assume that $\gcd(f(x),f(-x))$ is constant, so $f$ is
not 2-LRS-degenerate.  They try all candidate orders
$k=3,4,\ldots,5d^2$: for each candidate compute $R_{f,k} =
\res_y(f(y), y^k-x)$, and test whether $R_{f,k}$ is square-free; if not
then $f$ is $k$-LRS-degenerate.  We note that $R_{f,k} = G_k(f)$, the
$k$-th Graeffe transform of $f$.

\begin{remark}
For some indexes $k$ we have that $\deg(\Phi_k) \ll k$, so we can consider
computing $R^*_{f,k} = \res_y(f(xy), \Phi_k(y))$ and checking
whether it is square-free.  For such $k$ the advantage is that $R^*_{f,k}$ has rather lower degree
than $R_{f,k}$.  In practice, the cost of computing $R^*_{f,k}$ was sometimes higher than
the cost of computing $R_{f,k}$ because $y^k-x$ has a simple, sparse structure whereas $\Phi_k(y)$
is typically not sparse.
\end{remark}

\begin{remark}
The upper bound for the loop, $5d^2$, is actually a bound on the
inverse of Euler's totient function for the degrees they were able to
handle (see \cite{BD89}).  The basis for Algorithm~1 of \cite{CDM11}
tells us that if $f$ is $k$-LRS-degenerate then $\Phi_k \mid R_f$, thus
by considering degrees, $\totient(k) \le d^2-d$; whence the upper
bound (for $d \le 75$).  As we already mentioned, sequence \texttt{A355667} at
OEIS~\cite{OEIS} implies a dynamic bound, which is needed when $d > 75$.
Our Section~\ref{sec:conjectures} introduces a conjecture which, if true,
would let us use a much lower bound.
\end{remark}

\subsection{The modular algorithm}
\label{subsec:ModularAlgorithm}

We now present our modular approach whose existence was hinted at
in~\cite{CDM11}.  Our approach employs ``sophisticated
brute force'', and comprises two parts: one testing whether a
polynomial is $k$-LRS-degenerate for a specific order $k$, the other
choosing which orders $k$ to test.  Recall that the input is
content-free, square-free $f \in \ZZ[x]$ with $f(0) \neq 0$,
degree $\ge 2$ and $\gcd(f(x),f(-x)) = 1$.  The following
subsections give further details and justifications for our new
Algorithm~\ref{algm:LRSDegeneracyOrders}~(\textbf{LRSDegeneracyOrders}).

\begin{remark}
In Algorithm~\ref{algm:LRSDegeneracyOrders}~(\textbf{LRSDegeneracyOrders})
  the loop from Step~\ref{step:prime-amount} to Step~\ref{step:loop-end}
  performs~$3$ iterations because that gave a good compromise
  between speed of computation and exclusion of false positives.
\end{remark}

\subsection{The modular algorithm: supporting lemma}
\label{sec:modular-algm-supporting-lemma}

\begin{lemma}
\label{lemma:fieldhom}
Let $f \in \ZZ[x]$ be $k$-LRS-degenerate for some order $k \in \NN_{\ge 2}$.
Let $\FF_q$ be a finite field containing a primitive $k$-th root of unity, $\zeta_k$.
Let $\psi:\ZZ \to \FF_q$ be the canonical ring homomorphism,
with natural extension to $\psi: \ZZ[x] \to \FF_q[x]$.
Let $g = \psi(f)$, and suppose that $\deg(g) = \deg(f)$ then
$\deg \bigl(\gcd(g(x), g(\zeta_k x)) \bigr) \ge 1$ in $\FF_q[x]$.
\end{lemma}

\begin{proof}
Let $\ZZ[\xi_k]$ be the ring extension of $\ZZ$ by a $k$-th
root of unity, $\xi_k$.  We have a unique ring homomorphism $\theta: \ZZ[\xi_k] \to \FF_q$ sending
$\xi_k$ to $\zeta_k$, which extends naturally to $\theta: \ZZ[\xi_k][x] \to \FF_q[x]$; observe that $\theta|_{\ZZ[x]} = \psi$.  Regarding $\ZZ[x]$ as embedded in the ring
$\ZZ[\xi_k][x]$, let $h = \gcd(f(x), f(\xi_k x))$, so by assumption $\deg(h) \ge 1$.  Since $\deg(g)=\deg(f)$, we have also $\deg(\theta(h)) = \deg(h)$.  By construction $h(x) \mid f(x)$ and $h(x) \mid f(\xi_k x)$, so
$\theta(h(x)) \mid \theta(f(x))$ and $\theta(h(x)) \mid \theta(f(\xi_k x))$.
Now $\theta(f(x)) = g(x)$ and $\theta(f(\xi_k x)) = g(\zeta_k x)$; thus
$\theta(h(x)) \mid \gcd(g(x), g(\zeta_k x))$.
\end{proof}

We are especially interested in the case where $q = 1+rk$ is prime.
In this case, all primitive $k$-th roots of unity are elements of $\FF_q$.
Let $Z \subset \FF_q$ be the set of these primitive roots.  By the lemma, for each $\zeta_k \in Z$ we
have $\deg \bigl(\gcd(g(x),\, g(\zeta_k x)) \bigr) \ge 1$ in $\FF_q[x]$.

\subsection{Test for \texorpdfstring{$k$}{k}-LRS-degeneracy}

Here we present our sub-algorithm for testing whether $f$ is
$k$-LRS-degenerate: this corresponds to the loop controlled by Step~\ref{step:loop-test-lrsd}.
We assume that $k > 2$ since 2-LRS-degeneracy is easy to test.  Our
test is one-sided: it may produce a \emph{false positive}, namely,
report that $f$ is $k$-LRS-degenerate when it is actually not.
However, if $f$ is not $k$-LRS-degenerate, this is likely to be
detected quickly.

A candidate order $k$ can be fully verified by checking that
$\res_y(f(y), \Phi_k(xy))$ is not square-free; we did not put this
verification into the algorithm because it can be quite costly when
$k$ is large.  In Section~\ref{algm:CDM2} when discussing the second
algorithm from~\cite{CDM11}, we observed that computing $R_{f,k} =
\res_y(f(y),\, y^k-x)$ is typically faster than computing
$\res_y(f(y),\, \Phi_k(xy))$; we cannot take this ``short-cut'' here
because the non-square-freeness of $R_{f,k}$ indicates only that $f$ is
$\kappa$-LRS-degenerate for some factor $\kappa$ of $k$.

\begin{algorithm}[!ht]
    \caption{\textbf{(LRSDegeneracyOrders)}}
    \label{algm:LRSDegeneracyOrders}
    \begin{algorithmic}[1]
        \Require Non-constant polynomial $f \in \ZZ[x]$, square-free with $f(0) \neq 0$
        \Ensure A list $L$ of candidate LRS-degeneracy orders of $f$
        \Statex
        \State $K\gets$ list of candidate orders to test \label{step:k-to-test} \Comment{see Section~\ref{subsec:WhichKToTry}}
        \State $L \gets [\,]$, it will contain all ``probable'' orders detected
        \ForEach{candidate order $k \in K$} \label{step:loop-test-lrsd}
        \For{$i=1,2,3$} \label{step:prime-amount}
        \State Pick a prime $p$ with $p \equiv 1 \pmod{k}$ \label{step:prime-1-modn}
        \State Compute $\zeta_k$, a primitive $k$-th root of unity in $\FF_p$
        \For{$j=1,2,\ldots, \lfloor k/2 \rfloor$} \label{step:loop-roots}
        \IIf{$\gcd(j,k)=1$ and $\gcd(\bar{f}(x),\bar{f}(\zeta_k^j x))=1$}
        \State skip to next $k$\EndIIf \Comment{$\bar{f}$ is canonical image of $f$ in $\FF_p[x]$}
        \EndFor
        \EndFor  \label{step:loop-end}
        \State Append $k$ to $L$ as a probable LRS-degeneracy order
        \EndForEach
        \State \Return $L$
    \end{algorithmic}
\end{algorithm}

In Step~\ref{step:prime-1-modn} we pick a suitable finite field and
employ Lemma~\ref{lemma:fieldhom}.  By Dirichlet's theorem we know
that there are infinitely many primes of the form $1 +ks$.  To lower
the risk of false positives we choose a prime $p > 8 (\deg f)^2$ and
such that $s \ge 64$; we avoid small $s$ so that the vast majority of
field elements are not $k$-th roots of unity, and we want all ratios
of roots of $f$ to cover only a small proportion of the field
elements.  We obtain a primitive $k$-th root of unity via the hint in
Example~\ref{ex:PrimKthRoot}.  Our implementation chooses the primes
randomly, so it is harder to construct ``pathological'' inputs.


In the loop at Step~\ref{step:loop-roots} we can stop iterating at $\lfloor k/2 \rfloor$ because
for $j > k/2$ we have $\zeta_k^j = 1/\zeta_k^{k-j}$,
so letting $g = \gcd(\bar{f}(x), \bar{f}(\zeta_k^{k-j} x))$ we deduce that
$\gcd(\bar{f}(x), \bar{f}(\zeta_k^j x)) = u\, g(\zeta_k^j x)$ for some unit $u$.
So both gcds have the same degree.  Here we write $\bar{f}$ to denote
the canonical image of $f$ in $\FF_p[x]$.

To further reduce the chance of a false positive, we perform the
analogous test for up to two more primes: the loop at Step~\ref{step:prime-amount} performs
at most 3 iterations.  If all these checks pass then ``very likely''
$f$ is indeed $k$-LRS-degenerate, so we return a positive result.
Quantifying the probability of a false positive would require assuming
a distribution on the input polynomials, but there is no obvious,
reasonable choice.

\subsection{Which \texorpdfstring{$k$}{k} values to try?}
\label{subsec:WhichKToTry}

In Step~\ref{step:k-to-test} of Algorithm~\ref{algm:LRSDegeneracyOrders}~(\textbf{LRSDegeneracyOrders}) we need an
initial set of candidate orders to test.  Recall that we assume that
$f$ is square-free in $\ZZ[x]$; we write $d =\deg(f)$.  One possible
set comprises all $k$ such that $\totient(k) \le d^2-d$, because
if $f$ is $k$-LRS-degenerate then $\Phi_k(x)$ divides $\res_y(f(xy),f(y))/(x-1)^d$
which has degree $d^2-d$. We can refine this set of $k$ values to try as follows.
Assume that $f$ is $k$-LRS-degenerate, and let $\alpha,\beta$ be roots of $f$ for which $\alpha/\beta = \zeta_k$.
We consider two cases:
\begin{itemize}
\item If there is an irreducible factor $g$ of $f$ having both
  $\alpha$ and $\beta$ as roots then the algebraic extension
  $L = \QQ(\alpha,\beta)$ has degree $d_\alpha d_\beta$ where
  $d_\alpha = \deg(g)$ and $d_\beta < d_\alpha$.  Since $L$ contains
  $\QQ(\zeta_k)$ we have that $\totient(k)  \mid d_\alpha d_\beta$.

\item Otherwise there are distinct irreducible factors $g_\alpha$ and
  $g_\beta$ being the minimal polynomials for $\alpha$ and $\beta$
  respectively.  Without loss of generality $\deg(g_\alpha) \le \deg(g_\beta)$.
  Let $d_\alpha = \deg(g_\alpha)$ and $d_\beta = \deg(\tilde{g}_\beta) \le \deg(g_\beta)$
  where $\tilde{g}_\beta$ is the minimal polynomial of $\beta$ over $\QQ(\alpha)$.
  By the same argument as before we have that $\totient(k) \mid d_\alpha d_\beta$.
\end{itemize}
Since $\totient(k)$ is even for all $k > 2$ we can exclude pairs
$(d_\alpha, d_\beta)$ where both are odd.  The only instances of the
second case which are not covered by the first case are when $d_\alpha
= d_\beta$.  So we need to compute only the instances of the first
case, and add to them the squares of all even numbers up to $d/2$.  We
just enumerate all possibilities for $d_\alpha$ and $d_\beta$, and
build a ``sieve'' containing all $d_\alpha d_\beta$ values and their
even factors.  Finally we create a list of all $k$ such that
$\totient(k)$ is in the sieve.  This is what Algorithm~\ref{algm:LRSOrderCandidates}~(\textbf{LRSOrderCandidates}) does.
Empirically, this seems to eliminate $50$--$75\%$ of the candidate orders in the naive set.

\begin{algorithm}[!ht]
\caption{\textbf{(LRSOrderCandidates)}}
\label{algm:LRSOrderCandidates}
\begin{algorithmic}[1]
    \Require $d \in \ZZ_{\ge 2}$ being the degree of the polynomial $f$ we are testing
    \Ensure A list of candidate LRS-degeneracy orders to check
    \Statex
    \State $D \gets [\,]$, this will be our list of values which $\totient(k)$ must divide
    \For{$\alpha = 1,2,\ldots, d$}
    \For{$\beta = 1,2,\ldots, \alpha-1$}
    \IIf{$\alpha \beta$ is even} append $\alpha\beta$ to $D$\EndIIf
    \EndFor
    \EndFor
    \For{$\alpha = 2,4,6,\ldots, d/2$}
    \State Append $\alpha^2$ to $D$
    \EndFor
    \State \Return list of all $k>2$ up to $\totient^{-1}_{max}(d^2-d)$ such that $\totient(k)$ divides an element of $D$
\end{algorithmic}
\end{algorithm}


\subsubsection{Conjectured short-cut}
\label{sec:conjectures}

If we are just interested in the smallest order $k$ such that $f$ is $k$-LRS-degenerate then,
\emph{dependent on a conjecture,} we can reduce the list of $k$ to try considerably.
\begin{itemize}
\item We conjecture that if $f$ is LRS-degenerate then the minimal order $k$ satisfies $\totient(k) \le \deg(f)$.
\item Furthermore, we conjecture that if $f$ is irreducible and LRS-degenerate
then it is $k$-LRS-degenerate for some order $k$ with $\totient(k) \mid \deg(f)$.
\end{itemize}

\begin{example}
  Let $f = \Phi_3 \Phi_5$, so $\deg(f) = 6$.  Then $f$ is obviously
  $k$-LRS-degenerate for $k=3$ and $k=5$.  In fact,
  $f$ is also $15$-LRS-degenerate,  yet $\totient(15) = 8 > \deg(f)$.
  So if we want to find \emph{all} orders $k$ then the
  conjectured limit does not apply.
\end{example}


\subsection{Testing for LRS-degeneracy via evaluation}

We considered using the method of Section~\ref{subsec:CycloFactorsByEval}
to identify the cyclotomic factors in $R_f$.
A key point is that, for any evaluation point $b$, we have
$\res_x(f(x), f(bx)) = R_f(b)$, where $R_f$ is the resultant from
Section~\ref{algm:CDM1}.  Observe that the numerical resultant
$\res_x(f(x), f(bx))$ can be computed directly and quickly without
having to compute explicitly $R_f(y)$ as a polynomial.
We can also use the algorithm of Section~\ref{subsec:WhichKToTry} to
supply an initial list of candidate indexes.

We tried this implementation, but it was slower than our modular
method of Section~\ref{subsec:ModularAlgorithm}: the cost of computing
the numerical resultants (in CoCoALib) was too high.

\subsection{Timings}
\label{sec:TimingsLRSDegenerate}

In Table~\ref{tbl:LRSDegeneracyOrders} we present some timings of our
implementation of Algorithm~\ref{algm:LRSDegeneracyOrders}
(\textbf{LRSDegeneracyOrders}).  We tried the three largest examples
from~\cite{CDM11}, each of degree 24; but each example took less than
0.1s (to compute \emph{all} LRS-degeneracy orders).  Instead we
generated some higher degree polynomials: four with uniform random
integer coefficients from $-2^{10}$ to $2^{10}$ (a~larger coefficient
range made essentially no difference); and three products of
cyclotomic polynomials which have many orders.  As expected, none of
the random polynomials was LRS-degenerate.


In none of our tests did we observe any false positive.
Also, for polynomials with larger LRS-degeneracy orders, the
costs of verification dominate: for instance, we did not perform verification
for the degree $1\,236$ polynomial as fifteen of the candidate orders are at least $22\,755$,
and the largest order is $103\,037$.

\clearpage
\begin{table}[!ht]
	\centering
	\caption{Timings for \texttt{LRSDegeneracyOrders}}
	\label{tbl:LRSDegeneracyOrders}
	\renewcommand{\arraystretch}{1.1}
	\begin{tabular}{|l|S[table-format=4.0]|S[table-format=4.2]|S[table-format=3.2]|}
		\hline
		Description & {Deg}   &  {Unverified} & {Verified}  \\
		\hline                          
		Random  & 25 &  0.02 & 0.02 \\
		Random  & 50 &  0.03 & 0.03 \\
		Random  & 100 &  0.19 & 0.20  \\
		Random  & 200 &  1.7 & 1.7 \\
		$\Phi_{51} \Phi_{65} \Phi_{77} $ & 140 & 0.8 & 35 \\
		$\Phi_{165} \Phi_{183} $ & 200 & 2.2 & 118 \\
		$\Phi_{123}\Phi_{185}\Phi_{217}\Phi_{299}\Phi_{319}\Phi_{323}$ & 1236 & 3243 & {---} \\
		\hline
	\end{tabular}
\end{table}

\begin{remark}
  For the verification we checked that $\res_y(f(xy), \Phi_k(y))$ has a repeated factor;
  this is essentially the technique used inside the loop in CDM Algorithm~2~---~see Section~\ref{algm:CDM2}.
  To lower the computational cost we used the fact that the resultant is ``multiplicative'',
  namely $\res(f, g_1 g_2) = \res(f,g_1) \, \res(f,g_2)$, and the classical product from
  Lemma~\ref{lem:cyclo-prods}~(a).  We do this because it is decidedly quicker to compute $\res_y(f(xy), y^k-1)$ rather than $\res_y(f(xy), \Phi_k(y))$.
\end{remark}

\section{Conclusion}

We have presented new, practical algorithms for:
\begin{itemize}
  \item testing whether a polynomial
in $\ZZ[x]$ is cyclotomic, and if so, for finding its index.
\item finding the indexes of all cyclotomic factors of a given polynomial in $\ZZ[x]$
  with a low likelihood of false positives.
\item testing whether a polynomial in $\ZZ[x]$ is LRS-degenerate, and if so,
  determining all LRS-degeneracy orders, with a low likelihood of false positives.
  Currently, final verification of the candidate orders dominates the computation time.
\end{itemize}
The algorithms presented here are implemented as part of in CoCoALib~\cite{CoCoALib} (from version 0.99856).
The software is free and open source; it is available from \url{https://cocoa.dima.unige.it/}

\appendix


\bigskip\goodbreak

\bibliographystyle{elsarticle-num}
\bibliography{2026-06-18-LRSDegenerate}

\goodbreak

\end{document}